\documentclass{amsart}

\usepackage{amsmath,amssymb,amsthm}
\usepackage{indentfirst,color}
\usepackage[colorlinks,citecolor=red,linkcolor=blue,urlcolor=cyan]{hyperref}

\numberwithin{equation}{section}

\theoremstyle{plain}%
\newtheorem{theorem}{Theorem}[section]%
\newtheorem{lemma}[theorem]{Lemma}%
\newtheorem{corollary}[theorem]{Corollary}%
\theoremstyle{remark}%
\newtheorem{definition}[theorem]{Definition}%
\newtheorem{remark}[theorem]{Remark}%

\newcommand{\BB}{\mathbb{B}}%
\newcommand{\CC}{\mathbb{C}}%
\newcommand{\DD}{\mathbb{D}}%
\newcommand{\RR}{\mathbb{R}}%
\newcommand{\Sp}{\mathbb{S}}%
\newcommand{\calF}{\mathcal{F}}%
\newcommand{\calO}{\mathcal{O}}%
\newcommand{\pd}{\partial}%
\newcommand{\ddbar}{i\partial\bar{\partial}}%
\newcommand{\inner}[1]{\langle#1\rangle}%

\renewcommand{\geq}{\geqslant}%
\renewcommand{\leq}{\leqslant}%
\renewcommand{\Re}{\operatorname{Re}}%

\begin{document}

\title[A Characterization of Pluriharmonic Functions]{A Characterization of Pluriharmonic Functions via $p$-Bergman Kernels and A Forelli-type Result}
\author{Wang Xu}
\address{School of Mathematics, Sun Yat-sen University, Guangzhou 510275, China}
\email{xuwang@amss.ac.cn; xuwang6@mail.sysu.edu.cn}

\begin{abstract}
Given a plurisubharmonic function $\varphi>-\infty$ on the unit ball $\BB^n$ and a constant $0<p\leq2$, Guan-Zhou's optimal $L^p$ extension theorem yields a sharp lower bound for the weighted $p$-Bergman kernel. We show that the equality holds at some point if and only if $\varphi$ is pluriharmonic on $\BB^n$. This provides an analogue of the equality part of Suita's conjecture. As an application, we prove a Forelli-type result. Let $\varphi$ be a real-valued function on $\BB^n$. If every slice of $\varphi$ through the origin is harmonic and if $\varphi$ is plurisubharmonic in a neighborhood of the origin, then $\varphi$ is pluriharmonic on $\BB^n$.
\end{abstract}



\thanks{The author is supported by National Key R\&D Program of China (No. 2024YFA1015200) and National Natural Science Foundation of China (No. 12501107).}

\maketitle

\section{Introduction}

The Bergman kernel is an important object in several complex variables and complex geometry.
Given a domain $\Omega$ in $\CC^n$ and an upper semi-continuous function $\varphi$ on $\Omega$, the weighted \textit{Bergman space} of $\Omega$ is defined by
\begin{equation*}
A^2(\Omega;e^{-\varphi}) = \left\{f\in\calO(\Omega): \int_\Omega |f|^2e^{-\varphi} d\lambda < +\infty \right\},
\end{equation*}
and the weighted \textit{Bergman kernel} of $\Omega$ is defined by 
\begin{equation*}
K_\Omega(z;e^{-\varphi}) = \sup\left\{ |f(z)|^2 : f\in A^2(\Omega;e^{-\varphi}), \int_\Omega |f|^2e^{-\varphi} d\lambda \leq 1 \right\}, \quad z\in\Omega.
\end{equation*}

In this article, we focus on the special case where $\Omega$ is the unit ball $\BB^n \subset \CC^n$ and $\varphi>-\infty$ is a plurisubharmonic function on $\BB^n$. By the celebrated optimal $L^2$ extension theorem (see B{\l}ocki \cite{Blocki} and Guan-Zhou \cite{GuanZhou12,GuanZhou15}), there exists a holomorphic function $F\in\calO(\BB^n)$ such that $F(0)=1$ and
\begin{equation} \label{Eq:OptExt}
\int_{\BB^n} |F|^2e^{-\varphi} d\lambda \leq \frac{\pi^n}{n!} e^{-\varphi(0)}.
\end{equation}
Note that ${\pi^n}/{n!}$ is precisely the volume of $\BB^n$. As a consequence,
\begin{equation} \label{Eq:KerLB}
K_{\BB^n}(0;e^{-\varphi}) \geq \frac{n!}{\pi^n}e^{\varphi(0)}.
\end{equation}

This lower bound for the Bergman kernel is \textit{sharp} in the sense that the uniform constant cannot be replaced by any greater one. Indeed, if $\varphi$ is pluriharmonic, i.e., $\ddbar\varphi\equiv0$, then there exists $u\in\calO(\BB^n)$ such that $\varphi=2\Re u$; it then follows that $K_{\BB^n}(0;e^{-\varphi}) = K_{\BB^n}(0)|e^{u(0)}|^2$ and \eqref{Eq:KerLB} becomes an equality. On the other hand, if $\varphi$ is strictly plurisubharmonic, i.e., $\varphi\in C^2$ and $\ddbar\varphi>0$, then one can construct an $F\in\calO(\BB^n)$ for which the inequality \eqref{Eq:OptExt} is strict (see, for example, \cite[Theorem 1.7]{XuZhou24}); hence the inequality \eqref{Eq:KerLB} is also strict. We refer the interested reader to \cite{Inayama} and \cite{LiuXu24} for quantitative relations between the strict positivity of $\ddbar\varphi$ and the sharper estimate in \eqref{Eq:OptExt}.

Consequently, whether equality holds in \eqref{Eq:KerLB} is governed by the positivity of the weight: it holds when $\ddbar\varphi\equiv0$ and fails when $\ddbar\varphi>0$, while the intermediate case is far from clear. It is therefore a natural problem to characterize the equality case in \eqref{Eq:KerLB}. The expectation is that
\medskip
\begin{center}\itshape
	the equality in \eqref{Eq:KerLB} holds precisely when $\varphi$ is pluriharmonic.
\end{center}
\medskip
In particular, the expectation predicts a \textit{rigidity} phenomenon: the equality at a single point forces the vanishing of $\ddbar\varphi$ on the whole ball.

This problem is a weighted counterpart of the famous \textit{Suita conjecture}, which was completely solved by Guan-Zhou \cite{GuanZhou12,GuanZhou15} (the inequality part for planar domains was solved by B{\l}ocki \cite{Blocki}). Let $\Omega$ be an open Riemann surface admitting a Green function. Suita \cite{Suita} conjectured that $\pi K_\Omega(z) \geq c_\beta(z)^2$, and equality holds at some point if and only if $\Omega$ is conformally equivalent to the unit disc $\DD$, possibly minus a closed polar set. Here, $K_\Omega$ and $c_\beta$ denote the Bergman kernel and the logarithmic capacity of $\Omega$ with respect to a local coordinate. Thus, both problems ask when the Bergman kernel attains its lower bound, but one concerns the base manifold, while the other concerns the weight function. We refer the interested reader to \cite{XuZhou26} and the references therein for further developments of the generalized Suita conjecture.\\

There is an affirmative answer to the above problem in the case of $n=1$.

\begin{theorem}[Liu-Xu {\cite[Theorem 1.4]{LiuXu24}}; see also Guan-Mi {\cite[Theorem 1.11]{GuanMi22}}] \label{Thm:Harm}
Let $\varphi>-\infty$ be a subharmonic function on the unit disc $\DD$. Then
\begin{equation*}
K_{\DD}(0;e^{-\varphi}) = \frac{1}{\pi} e^{\varphi(0)}
\end{equation*}
if and only if $\varphi$ is harmonic on $\DD$.
\end{theorem}

Among other things, its proof uses the fact that any $(1,1)$-current on $\DD\subset\CC$ is automatically closed. Therefore, the proof cannot be directly generalized to higher dimensions. In this article, we successfully extend Theorem \ref{Thm:Harm} to general dimensions, and thereby obtain a characterization of pluriharmonic functions via a Suita-type equality.

\begin{theorem} \label{Thm:CharPH}
Let $\varphi>-\infty$ be a plurisubharmonic function on $\BB^n$. Then
\begin{equation*}
K_{\BB^n}(0;e^{-\varphi}) = \frac{n!}{\pi^n} e^{\varphi(0)}
\end{equation*}
if and only if $\varphi$ is pluriharmonic on $\BB^n$.
\end{theorem}

Key ingredients of the proof include Guan-Zhou's \cite{GuanZhou15} optimal $L^2$ extension theorem (which we need to apply in three different settings), Guan's \cite{Guan19} concavity property of minimal $L^2$ integrals, and Xu-Zhou's \cite{XuZhou24} necessary condition for the linearity of minimal $L^2$ integrals. In particular, we develop a mechanism that transfers the condition on $\BB^n$ to each affine disc $\xi\DD$ passing through the origin, and then apply Theorem \ref{Thm:Harm}. Of course, a function that is harmonic on each slice $\xi\DD$ is not necessarily pluriharmonic (see the counterexample \eqref{Eq:CE} below). We use a rigidity property for minimal $L^2$ integrals to upgrade the slicewise harmonicity to pluriharmonicity.

From the viewpoint of its proof, Theorem \ref{Thm:CharPH} is closely connected with the classical Forelli theorem. Let $\varphi$ be a real-valued function on $\BB^n$. Forelli \cite{Forelli} proved that if every slice of $\varphi$ through the origin is harmonic and if $\varphi$ is smooth at the origin, then $\varphi$ is pluriharmonic on $\BB^n$. As an application of Theorem \ref{Thm:CharPH}, we prove a Forelli-type result by replacing the smoothness assumption with plurisubharmonicity.

\begin{theorem} \label{Thm:Forelli-type}
Let $\varphi$ be a real-valued function on $\BB^n$. Assume that the slice $w\mapsto\varphi(w\xi)$ is harmonic on $\DD$ for every $\xi\in\Sp^{2n-1}$. If $\varphi$ is plurisubharmonic in a neighborhood of $0$, then $\varphi$ is pluriharmonic on $\BB^n$.
\end{theorem}

This may be viewed as a generalization of Forelli's theorem in a direction distinct from that of other authors (see, for example, Joo-Kim-Schmalz \cite{JKS13, JKS16}). There is a classical counterexample (see \cite[\S 4.4.8]{Rudin}) showing that the plurisubharmonicity near $0$ is necessary: the function
\begin{equation} \label{Eq:CE}
	\varphi(z) = \frac{z_1^4\overline{z_2}+\overline{z_1}^4z_2}{|z|^2}
\end{equation}
is $C^2$-smooth on $\BB^2$; each slice $w\mapsto\varphi(w\xi)$ is harmonic, yet $\varphi(z)$ is not pluriharmonic. A direct computation gives
\begin{equation*}
	\det\left( \frac{\pd^2\varphi}{\pd z_j\pd\overline{z_k}} \right) = -\frac{9|z_1|^{10}}{|z|^8},
\end{equation*}
which shows that $\varphi(z)$ is not plurisubharmonic near $0$.

There has been increasing interest in $p$-Bergman theory in recent years (see, for example, \cite{NZZ16, ChenZhang} and the references therein). Let $\varphi$ be an upper semi-continuous function on a bounded domain $\Omega\subset\CC^n$. For any constant $p>0$, we define the weighted \textit{$p$-Bergman space} of $\Omega$ by
\begin{equation*}
A^p(\Omega;e^{-\varphi}) = \left\{f\in\calO(\Omega): \int_\Omega |f|^pe^{-\varphi} d\lambda < +\infty \right\},
\end{equation*}
and the weighted \textit{$p$-Bergman kernel} of $\Omega$ by
\begin{equation*}
	K_{\Omega,p}(z;e^{-\varphi}) = \sup\left\{ |f(z)|^p : f\in A^p(\Omega;e^{-\varphi}), \int_\Omega |f|^pe^{-\varphi} d\lambda \leq 1 \right\}, \quad z\in\Omega.
\end{equation*}

Now we consider the case where $0<p\leq2$ and $\varphi>-\infty$ is a plurisubharmonic function on $\BB^n$. According to Guan-Zhou's \cite{GuanZhou15} optimal $L^p$ extension theorem (see Theorem \ref{Thm:OptLpExt}), there exists a holomorphic function $f$ on $\BB^n$ such that $f(0)=1$ and
\begin{equation*}
\int_{\BB^n} |f|^pe^{-\varphi} d\lambda \leq \frac{\pi^n}{n!} e^{-\varphi(0)}.
\end{equation*}
Consequently,
\begin{equation*}
K_{\BB^n,p}(0;e^{-\varphi}) \geq \frac{n!}{\pi^n} e^{\varphi(0)}.
\end{equation*}
It is well-known that $\BB^n$ is homogeneous, i.e., for any $z,w\in\BB^n$, there exists a holomorphic automorphism $\Phi$ of $\BB^n$ such that $\Phi(z)=w$. Using the transformation rule for weighted $p$-Bergman kernels, the above inequality becomes
\begin{equation} \label{Eq:pKerLB}
	K_{\BB^n,p}(z;e^{-\varphi}) \geq \frac{n!}{\pi^n} \frac{e^{\varphi(z)}}{(1-|z|^2)^{n+1}}, \quad z\in\BB^n.
\end{equation}

The inequality \eqref{Eq:pKerLB} is equivalent to an optimal $L^p$ extension result from $z\in\BB^n$ to $\BB^n$ (see Theorem \ref{Thm:OptLpExt-2}). As in the case of $p=2$, it is also a natural question to ask when equality holds in \eqref{Eq:pKerLB}. Based on Theorem \ref{Thm:CharPH}, we have the following characterization of the equality case: the equality in \eqref{Eq:pKerLB} holds at some point if and only if $\varphi$ is pluriharmonic on $\BB^n$.

\begin{theorem} \label{Thm:CharPH-pKer}
Let $\varphi>-\infty$ be a plurisubharmonic function on $\BB^n$ and $0<p\leq 2$ a constant. Then
\begin{equation*}
	K_{\BB^n,p}(z;e^{-\varphi}) \geq \frac{n!}{\pi^n} \frac{e^{\varphi(z)}}{(1-|z|^2)^{n+1}}, \quad z\in\BB^n,
\end{equation*}
and equality holds at some point if and only if $\varphi$ is pluriharmonic on $\BB^n$.
\end{theorem}

The rest of the article is structured as follows. In Section \ref{Sec2}, we recall some preparatory results. In Sections \ref{Sec3}, \ref{Sec4} and \ref{Sec5}, we prove Theorems \ref{Thm:CharPH}, \ref{Thm:Forelli-type} and \ref{Thm:CharPH-pKer}, respectively. \\

\textbf{Acknowledgements.} The author would like to thank Prof. Xiangyu Zhou for introducing him to the $L^2$ extension theory and for his unwavering encouragement and support over the years.

\textbf{Use of artificial intelligence.} The author employed AI (\verb|deepseek-v4-pro|) for language polishing and checking the arguments. All mathematical contents were developed solely by the author.

\section{Preliminaries} \label{Sec2}

\subsection{Optimal $L^2$ Extension Theorems}

We recall a simplified version of Guan-Zhou's \cite{GuanZhou15} optimal $L^2$ extension theorem with gain, in the case of holomorphic functions on pseudoconvex domains.

\begin{definition}
We call a smooth positive function $c(t)>0$ on $(-\infty,0)$ a \textit{gain} if the following conditions are satisfied:
\begin{gather}
\label{Eq:Cond1} \varliminf\limits_{t\to-\infty}c(t)e^{-t}>0, \\
\label{Eq:Cond2} \int_{-\infty}^0c(t)dt<+\infty, \\
\label{Eq:Cond3}
\left(\int_t^0c(\tau)d\tau\right)^2 > c(t)\int_t^0 \left(\int_{\tau_1}^0c(\tau_2)d\tau_2\right) d\tau_1, \quad \forall t<0.
\end{gather}
\end{definition}

\begin{remark} \label{Rmk:Gain}
Let $c(t)>0$ be a smooth function on $(-\infty,0)$ satisfying \eqref{Eq:Cond1} and \eqref{Eq:Cond2}.
If $c(t)$ is increasing, then inequality \eqref{Eq:Cond3} holds. Moreover, if there exists some $t_0<0$ such that $c'\geq0$ on $(-\infty,t_0)$, but $c'<0$ and $(\log c)''<0$ on $(t_0,0)$, then inequality \eqref{Eq:Cond3} also holds (see \cite[Remark 4.12]{GuanZhou15} for details).
\end{remark}

\begin{theorem}[Guan-Zhou] \label{Thm:OptExt}
Let $\Omega$ be a pseudoconvex domain in $\CC_z^n=\CC_{z'}^{n-k}\times\CC_{z''}^k$ and $\phi$ a plurisubharmonic function on $\Omega$. Assume that there exists a plurisubharmonic function $\psi<0$ on $\Omega$ such that $\rho(z):=\psi(z)-2k\log|z''|$ is continuous on $\Omega$ and smooth outside $S:=\Omega\cap\{z''=0\}$. Let $c(t)>0$ be a smooth gain on $(-\infty,0)$.
Then for any holomorphic function $f$ on $S$ satisfying
\begin{equation*}
\int_S|f|^2e^{-\phi-\rho}d\lambda< +\infty,
\end{equation*}
there exists a holomorphic function $F$ on $\Omega$ such that $F|_S=f$ and
\begin{equation*}
\int_\Omega|F|^2c(\psi)e^{-\psi-\phi}d\lambda \leq \frac{\pi^k}{k!} \int_{-\infty}^0c(t)dt \int_S|f|^2e^{-\phi-\rho}d\lambda.
\end{equation*}
\end{theorem}

\begin{remark}
In the above theorem, $e^{-\rho}$ is the density of the Ohsawa measure associated to $\psi$ (see \cite{GuanZhou15} for details): for any compactly supported continuous function $g$ on $\Omega$, we have
\begin{equation*}
\lim_{t\to+\infty} \int_{\{-t-1<\psi<-t\}} ge^{-\psi} d\lambda = \lim_{a\to+\infty} e^a \int_{\{\psi<-a\}} g d\lambda = \frac{\pi^k}{k!} \int_S ge^{-\rho} d\lambda.
\end{equation*}
\end{remark}

Throughout this article, we denote a general disc in $\CC$ by
\begin{equation*}
	\DD(z_0,R) = \{z\in\CC: |z-z_0|<R\}
\end{equation*}
and a general ball in $\CC^n$ by
\begin{equation*}
	\BB^n(z_0,R) = \{z\in\CC^n: |z-z_0|<R\}.
\end{equation*}
Therefore, $\DD=\DD(0,1)$ and $\BB^n=\BB^n(0,1)$. Moreover, we denote by $\Sp^{2n-1}$ the unit sphere in $\CC^n$.

In the proof of Theorem \ref{Thm:CharPH}, we need three special cases of Theorem \ref{Thm:OptExt}.

\begin{corollary} \label{Cor:OptExt1}
Let $\phi$ be a plurisubharmonic function on $\BB^n(0,R)$ with $\phi(0)>-\infty$. Then there exists a holomorphic function $F$ on $\BB^n(0,R)$ such that $F(0)=1$ and
\begin{equation*}
	\int_{\BB^n(0,R)} |F|^2e^{-\phi} d\lambda \leq \frac{\pi^n}{n!} R^{2n} e^{-\phi(0)}.
\end{equation*}
\end{corollary}

\begin{proof}
We take $\psi(z) = 2n\log\dfrac{|z|}{R}$, which is a negative plurisubharmonic function on $\BB^n(0,R)$. Clearly, $S=\{0\}$ and $\rho=-2n\log R$. The corollary then follows by applying Theorem \ref{Thm:OptExt} with the gain $c(t)=e^t$.
\end{proof}

\begin{corollary} \label{Cor:OptExt2}
Let $\phi$ be a plurisubharmonic function on $\BB^n(0,R)$, where $n\geq2$. For any holomorphic function $f$ on $\DD(0,R)$ satisfying
\begin{equation*}
\int_{\DD(0,R)} |f(z_1)|^2 (R^2-|z_1|^2)^{n-1} e^{-\phi(z_1,0,\cdots,0)} d\lambda < +\infty,
\end{equation*}
there exists a holomorphic function $F$ on $\BB^n(0,R)$ such that $F(z_1,0,\cdots,0)=f(z_1)$ on $\DD(0,R)$ and
\begin{equation*}
\int_{\BB^n(0,R)} |F|^2e^{-\phi} d\lambda \leq \frac{\pi^{n-1}}{(n-1)!} \int_{\DD(0,R)} |f(z_1)|^2 (R^2-|z_1|^2)^{n-1} e^{-\phi(z_1,0,\cdots,0)} d\lambda.
\end{equation*}
\end{corollary}

\begin{proof}
We write the coordinates of $\CC^n$ as $(z_1,z'')$. Since
\begin{equation*}
\frac{\pd^2}{\pd z_1\pd\bar{z}_1} \Big(-\log(R^2-|z_1|^2)\Big) = \frac{R^2}{(R^2-|z_1|^2)^2} > 0,
\end{equation*}
we know $\psi(z) = (n-1)\log\dfrac{|z''|^2}{R^2-|z_1|^2}$ is a negative plurisubharmonic function on $\BB^n(0,R)$. We have $\rho=-\log(R^2-|z_1|^2)^{n-1}$ and $S=\{(z_1,0,\cdots,0):|z_1|<R\}$. Then the corollary is an application of Theorem \ref{Thm:OptExt} with the gain $c(t)=e^t$.
\end{proof}

\begin{corollary} \label{Cor:OptExt3}
Let $\phi$ be a subharmonic function on $\DD(0,R)$ with $\phi(0)>-\infty$. For any positive integer $n\geq2$, there exists a holomorphic function $f$ on $\DD(0,R)$ such that $f(0)=1$ and
$$ \int_{\DD(0,R)} |f(w)|^2 (R^2-|w|^2)^{n-1} e^{-\phi(w)} d\lambda \leq \frac{\pi}{n}R^{2n} e^{-\phi(0)}. $$
\end{corollary}

\begin{proof}
We take $\psi(z) = 2\log\dfrac{|z|}{R} < 0$. Then $S=\{0\}$ and $\rho=-2\log R$. Consider a positive smooth function $c(t)=(1-e^t)^{n-1}e^t$ on $(-\infty,0)$. Clearly,
\begin{equation*}
	\int_{-\infty}^0 c(t) dt = \int_0^1 (1-s)^{n-1} ds = \frac{1}{n}.
\end{equation*}
Since
\begin{equation*}
c'(t)=(1-ne^t)(1-e^t)^{n-2}e^t \quad\text{and}\quad (\log c)''=-\frac{(n-1)e^t}{(1-e^t)^2},
\end{equation*}
we see that $c(t)$ satisfies the second criterion in Remark \ref{Rmk:Gain}. 
Then the corollary is an application of Theorem \ref{Thm:OptExt} with the gain $c(t)$.
\end{proof}

By an iterative method, Guan-Zhou \cite{GuanZhou15} also proved an optimal $L^p$ extension theorem for $0<p<2$. In this paper, we need the following special case.

\begin{theorem}[Guan-Zhou] \label{Thm:OptLpExt}
Let $\phi$ be a plurisubharmonic function on $\BB^n(0,R)$ with $\phi(0)>-\infty$. Let $0<p\leq2$ be a constant. Then there exists a holomorphic function $F$ on $\BB^n(0,R)$ such that $F(0)=1$ and
\begin{equation*}
	\int_{\BB^n(0,R)} |F|^pe^{-\phi} d\lambda \leq \frac{\pi^n}{n!} R^{2n} e^{-\phi(0)}.
\end{equation*}
\end{theorem}

\subsection{Minimal $L^2$ Integrals}

We recall some properties of minimal $L^2$ integrals.

Let $\Omega$ be a domain in $\CC^n$ and $\varphi$ an upper semi-continuous function on $\Omega$. It is well-known that $A^2(\Omega;e^{-\varphi})$ is a Hilbert space whose inner product structure is given by
\begin{equation*}
\langle f,g\rangle := \int_\Omega f\bar{g} e^{-\varphi} d\lambda, \quad f,g\in A^2(\Omega;e^{-\varphi}).
\end{equation*}
We denote $\|f\| = \sqrt{\langle f,f\rangle}$. Fix a point $z_0\in\Omega$. We assume that there exists some $F\in A^2(\Omega;e^{-\varphi})$ with $F(z_0)=1$. We consider the following minimal $L^2$ integral
\begin{equation} \label{Eq:Inf}
	I := \inf\left\{ \int_\Omega |f|^2e^{-\varphi} d\lambda: f\in A^2(\Omega;e^{-\varphi}), f(z_0)=1 \right\}.
\end{equation}

\begin{lemma}
There exists a unique minimizer for the infimum \eqref{Eq:Inf}.
\end{lemma}

\begin{proof}
We denote by $\calF$ the collection of all holomorphic functions $f\in A^2(\Omega;e^{-\varphi})$ with $f(z_0)=1$. Since $F\in\calF$, the collection is nonempty. We take a sequence $\{f_j\}$ in $\calF$ such that
\begin{equation*}
	\lim_{j\to+\infty} \int_\Omega |f_j|^2e^{-\varphi} d\lambda = I.
\end{equation*}
Since $\varphi$ is locally bounded from above, we see that $\int_K |f_j|^2 d\lambda$ is uniformly bounded for any compact subset $K$ of $\Omega$. According to Montel's theorem, there exists a subsequence $\{f_{j_k}\}$ which converges compactly to some $\hat{f}\in\calO(\Omega)$. Clearly,
\begin{equation*}
\hat{f}(z_0) = \lim_{k\to+\infty} f_{j_k}(z_0) = 1
\end{equation*}
and
\begin{align*}
I \leq \int_\Omega |\hat{f}|^2e^{-\varphi} d\lambda & = \int_\Omega \lim_{k\to+\infty} |f_{j_k}|^2e^{-\varphi} d\lambda \\
& \leq \varliminf_{k\to+\infty} \int_\Omega |f_{j_k}|^2e^{-\varphi} d\lambda = I.
\end{align*}
Therefore, $\hat{f}$ is a minimizer for the infimum \eqref{Eq:Inf}.

Assume that $\hat{g}$ is another minimizer for the infimum \eqref{Eq:Inf}. Then $(\hat{f}+\hat{g})/2 \in \calF$. It is clear that
\begin{align*}
I \leq \Big\|\frac{\hat{f}+\hat{g}}{2}\Big\|^2 & \leq \Big\|\frac{\hat{f}+\hat{g}}{2}\Big\|^2 + \Big\|\frac{\hat{f}-\hat{g}}{2}\Big\|^2 \\
& = \frac{\|\hat{f}\|^2 + \|\hat{g}\|^2}{2} = I.
\end{align*}
Consequently, $(\hat{f}-\hat{g})/2 \equiv 0$. This proves the uniqueness of the minimizer.
\end{proof}

From another perspective, the minimizer is precisely the minimal $L^2$ extension of the constant function $1$ from $S=\{z_0\}$ to $\Omega$.

\begin{lemma} \label{Lemma:Ortho}
A function $f \in A^2(\Omega;e^{-\varphi})$ is the minimizer for the infimum \eqref{Eq:Inf} if and only if
\begin{equation*}
	\int_\Omega f\overline{g}e^{-\varphi} d\lambda = 0
\end{equation*}
for any $g\in A^2(\Omega;e^{-\varphi})$ with $g(z_0)=0$.
\end{lemma}

\begin{proof}
Assume that $f \in A^2(\Omega;e^{-\varphi})$ is the minimizer. For any $g\in A^2(\Omega;e^{-\varphi})$ with $g(z_0)=0$, we have
\begin{equation*}
\|f\|^2 \leq \|f+\tau g\|^2 = \|f\|^2 + 2\Re(\tau\langle g,f\rangle) + |\tau|^2 \|g\|^2, \quad \forall\tau\in\CC.
\end{equation*}
Choose $\tau=-t\langle f,g\rangle$ with $t>0$. Then the inequality becomes
\begin{equation*}
- 2t |\langle f,g\rangle|^2 + t^2 |\langle f,g\rangle|^2 \|g\|^2\geq 0, \quad \forall t>0.
\end{equation*}
Dividing by $t$ and letting $t\to0^+$, we obtain $-2|\langle f,g\rangle|^2 \geq0$. Therefore,
\begin{equation*}
\langle f,g\rangle = \int_\Omega f\overline{g}e^{-\varphi} d\lambda = 0.
\end{equation*}

Conversely, assume that the above equality holds for any $g\in A^2(\Omega;e^{-\varphi})$ with $g(z_0)=0$. Consider another function $h\in A^2(\Omega;e^{-\varphi})$ with $h(z_0)=1$. Then
\begin{align*}
\|h\|^2 & = \|f+(h-f)\|^2 \\
& = \|f\|^2 + 2\Re(\langle f,h-f\rangle) + \|h-f\|^2 \\
& \geq \|f\|^2 + \|h-f\|^2 \geq \|f\|^2.
\end{align*}
This completes the proof.
\end{proof}

Next, we consider the variational property of minimal $L^2$ integrals on sublevel sets of a plurisubharmonic function.

Let $\Omega$ be a pseudoconvex domain in $\CC^n$ and $\varphi$ a plurisubharmonic function on $\Omega$. Assume that there exists a plurisubharmonic function $\psi<0$ on $\Omega$ such that $e^{-\psi}$ is not integrable at a given point $z_0\in\Omega$. In particular, $\psi(z_0)=-\infty$.

Assume that there exists a holomorphic function $F\in\calO(\Omega)$ with $F(z_0)=1$ and
\begin{equation*}
\int_\Omega |F|^2e^{-\varphi} d\lambda < +\infty.
\end{equation*}
For each $t\leq0$, let $\Omega_t = \{z\in\Omega: \psi(z)<t\}$ be the sublevel set of $\psi$. We consider the following minimal $L^2$ integral
\begin{equation*}
I(t) := \inf\left\{ \int_{\Omega_t} |f|^2e^{-\varphi} d\lambda: f\in\calO(\Omega_t), f(z_0)=1 \right\}.
\end{equation*}
We also denote by $F_t\in \calO(\Omega_t)$ the unique minimizer for the infimum defining $I(t)$.
Clearly, $0 < I(t) < +\infty$, $I(t)$ is increasing with respect to $t$, and
\begin{equation*}
0 \leq \lim_{t\to-\infty} I(t) \leq \lim_{t\to-\infty} \int_{\Omega_t} |F|^2e^{-\varphi} d\lambda = 0.
\end{equation*}

Guan \cite{Guan19} proved a useful concavity property for minimal $L^2$ integrals.

\begin{theorem}[{\cite[Proposition 4.1]{Guan19}}] \label{Thm:Concave}
$I(\log r)$ is a concave function of $r\in(0,1]$.
\end{theorem}

Moreover, Xu-Zhou \cite{XuZhou24} obtained a necessary condition under which the concavity of minimal $L^2$ integrals reduces to linearity.

\begin{theorem}[{\cite[Remark 5.3]{XuZhou24}}] \label{Thm:Linear}
If $I(\log r)$ is a linear function of $r\in(0,1]$, then
\begin{equation*}
	F_t \equiv F_0|_{\Omega_t}, \quad \forall t<0.
\end{equation*}
\end{theorem}

In other words, when $r\mapsto I(\log r)$ is linear, the minimal $L^2$ extension $F_t$ on each sublevel set $\Omega_t$ is the restriction of a single holomorphic function $F_0$. This rigidity plays an important role in the proof of Theorem \ref{Thm:CharPH}.

We refer the interested readers to Guan-Mi \cite{GuanMi22} and their subsequent work for more general results on the concavity and linearity of minimal $L^2$ integrals.

\subsection{Forelli's Theorem \& A Lemma From Real Analysis}

\begin{theorem}[{Forelli \cite{Forelli}}] \label{Thm:Forelli}
Let $\phi$ be a real-valued function on $\BB^n$ such that $w\mapsto\phi(w\xi)$ is harmonic on $\DD$ for any $\xi\in\Sp^{2n-1}$. If for every $k\in\mathbb{N}$ there exists an open neighborhood $U_k\ni0$ such that $\phi\in C^k(U_k)$, then $\phi$ is pluriharmonic on $\BB^n$. 
\end{theorem}

\begin{lemma}
Let $\rho$ be a measurable function on $(0,1)$ such that $\int_0^r |\rho(s)| ds <+\infty$ for any $0<r<1$. Let $n$ be a positive integer. If
\begin{equation*}
	\int_0^r \rho(s) (r-s)^{n-1} ds = 0 
\end{equation*}
for any $0<r<1$, then $\rho=0$ almost everywhere on $(0,1)$.
\end{lemma}

\begin{proof}
We shall prove the lemma by induction on $n$. If $n=1$, then the lemma is a direct consequence of Lebesgue's differentiation theorem.

Now, we assume that the lemma has been proved for $n=k$, and we turn to the case of $n=k+1$. We consider
\begin{equation*}
\eta(r) = \int_0^r \rho(s) (r-s)^k ds
\end{equation*}
as a function of $r\in(0,1)$. Let $r\in(0,1)$ be fixed and $\Delta r>0$ be small enough so that $r+\Delta r<1$. Then
\begin{align*}
\frac{\eta(r+\Delta r)-\eta(r)}{\Delta r} & = \frac{1}{\Delta r} \int_r^{r+\Delta r} \rho(s) (r+\Delta r-s)^k ds \\
& + \int_0^r \rho(s) \frac{(r+\Delta r-s)^k-(r-s)^k}{\Delta r} ds =: \mathbf{I} + \mathbf{II}.
\end{align*}

Clearly,
\begin{equation*}
|\mathbf{I}| \leq (\Delta r)^{k-1} \int_r^{r+\Delta r} |\rho(s)| ds.
\end{equation*}
Since $k-1\geq0$ and $|\rho|$ is locally integrable near $r$, we have $\lim\limits_{\Delta r\to0} \mathbf{I} = 0$.

On the other hand,
\begin{equation*}
\mathbf{II} - k \int_0^r \rho(s) (r-s)^{k-1} ds = \sum_{j=2}^k
\begin{pmatrix} k \\ j \end{pmatrix}
\int_0^r \rho(s) (r-s)^{k-j}(\Delta r)^{j-1} ds,
\end{equation*}
and then
\begin{equation*}
\left| \mathbf{II} - k \int_0^r \rho(s) (r-s)^{k-1} ds \right| \leq \sum_{j=2}^k
\begin{pmatrix} k \\ j \end{pmatrix}
 r^{k-j}(\Delta r)^{j-1} \int_0^r |\rho(s)| ds.
\end{equation*}
As $\int_0^r |\rho(s)| ds < +\infty$, the right hand side converges to zero as $\Delta r\to0$. Therefore,
\begin{equation*}
0 = \lim_{\Delta r\to0} \frac{\eta(r+\Delta r)-\eta(r)}{\Delta r} = k \int_0^r \rho(s) (r-s)^{k-1} ds
\end{equation*}
for any fixed $0<r<1$. By the inductive hypothesis, $\rho=0$ almost everywhere on $(0,1)$.
\end{proof}

After a change of variable, we have the following corollary.

\begin{corollary} \label{Cor:RA}
Let $\eta$ be a measurable function on $(0,1)$ such that $\int_0^r |\eta(s)|s ds <+\infty$ for any $0<r<1$. Let $n$ be a positive integer. If
\begin{equation*}
	\int_0^r \eta(s) (r^2-s^2)^{n-1} sds = 0 
\end{equation*}
for any $0<r<1$, then $\eta=0$ almost everywhere on $(0,1)$.
\end{corollary}

\section{Proof of Theorem \ref{Thm:CharPH}} \label{Sec3}

We only need to consider the case of $n\geq2$. The sufficiency part of Theorem \ref{Thm:CharPH} is obvious. In what follows, we divide the proof of the necessity part into several lemmas.

Consider the negative plurisubharmonic function $\psi = 2n\log|z|$ on $\BB^n$. It is clear that $e^{-\psi}$ is not integrable at $z=0$. For each $t\leq 0$, we denote by $\Omega_t = \{\psi<t\}$ the sublevel set of $\psi$ and
\begin{equation*}
	I(t) = \inf\left\{ \int_{\Omega_t} |F|^2e^{-\varphi}d\lambda: F\in\calO(\Omega_t),F(0)=1 \right\}
\end{equation*}
the minimal $L^2$ integral. For each $t\leq0$, we denote by $F_t\in\calO(\Omega_t)$ the unique minimizer for the infimum defining $I(t)$. In other words, $F_t$ is the minimal $L^2$ extension of the constant function $1$ from $S=\{0\}$ to $\Omega_t = \BB^n(0,e^{t/2n})$.

\begin{lemma} \label{Lemma1}
For any $t\leq 0$, we have $F_t \equiv F_0|_{\Omega_t}$ and
\begin{equation} \label{Eq:It}
I(t) = \int_{\Omega_t} |F_t|^2e^{-\varphi} d\lambda = \frac{\pi^n}{n!}e^{-\varphi(0)+t}.
\end{equation}
\end{lemma}

\begin{proof}
According to Guan-Zhou's optimal $L^2$ extension theorem (Corollary \ref{Cor:OptExt1}), for each $t\leq0$, there exists a holomorphic function $f_t$ on the ball $\BB^n(0,e^{t/2n}) = \Omega_t$ such that $f_t(0)=1$ and
\begin{equation*}
	\int_{\Omega_t} |f_t|^2e^{-\varphi} d\lambda \leq \frac{\pi^n}{n!}e^{-\varphi(0)+t}.
\end{equation*}
Consequently,
\begin{equation} \label{Eq:UBIt}
	I(t) \leq \frac{\pi^n}{n!}e^{-\varphi(0)+t}, \quad t\leq0.
\end{equation}

By Guan's concavity property for minimal $L^2$ integrals (Theorem \ref{Thm:Concave}),
\begin{equation*}
J(r) := I(\log r)
\end{equation*}
is a concave increasing function of $r\in(0,1]$. Since
\begin{equation*}
\lim_{r\to0^+} J(r) = \lim_{t\to-\infty} I(t) = 0,
\end{equation*}
the concavity implies that
\begin{equation} \label{Eq:LBJr}
	rJ(1) \leq J(r), \quad 0<r\leq 1.
\end{equation}

By the definition of Bergman kernels and the assumption of Theorem \ref{Thm:CharPH},
\begin{equation} \label{Eq:I0}
	I(0) = \frac{1}{K_{\BB^n}(0;e^{-\varphi})} = \frac{\pi^n}{n!}e^{-\varphi(0)}.
\end{equation}
Combining \eqref{Eq:UBIt}, \eqref{Eq:LBJr} and \eqref{Eq:I0}, we have
\begin{equation*}
	\frac{\pi^n}{n!}e^{-\varphi(0)}\cdot r = rJ(1) \leq J(r) \leq \frac{\pi^n}{n!}e^{-\varphi(0)}\cdot r.
\end{equation*}
It follows that
\begin{equation*} 
	J(r) \equiv \frac{\pi^n}{n!}e^{-\varphi(0)}\cdot r,
\end{equation*}
i.e., $J(r)$ is a linear function of $r\in(0,1]$. According to Xu-Zhou's necessary condition for the linearity of minimal $L^2$ integrals (Theorem \ref{Thm:Linear}),
\begin{equation*} 
	F_t \equiv F_0|_{\Omega_t}, \quad t<0.
\end{equation*}
This completes the proof.
\end{proof}

In the following, we will show that $w\mapsto\varphi(w\xi)$ is a harmonic function on $\DD$ for any $\xi\in\Sp^{2n-1}$. For convenience, we denote by
\begin{equation*}
F_\xi(w) = F_0(w\xi) \quad\text{and}\quad \varphi_\xi(w) = \varphi(w\xi)
\end{equation*}
the slices of $F_0\in\calO(\BB^n)$ and $\varphi\in\text{psh}(\BB^n)$ in the affine disc $\xi\DD = \{w\xi: w\in\DD\}$.

\begin{lemma} \label{Lemma2}
For any $\xi\in\Sp^{2n-1}$ and any $0<r\leq1$, the function $F_\xi|_{\DD(0,r)}$ is the unique minimizer for the following infimum:
\begin{equation} \label{Eq:Inf1}
\inf\left\{ \int_{\DD(0,r)} |g(w)|^2 (r^2-|w|^2)^{n-1} e^{-\varphi_\xi(w)} d\lambda: g\in\calO(\DD(0,r)), g(0)=1 \right\}.
\end{equation}
Moreover,
\begin{equation} \label{Eq:Fxi}
\int_{\DD(0,r)} |F_\xi(w)|^2 (r^2-|w|^2)^{n-1} e^{-\varphi_\xi(w)} d\lambda = \frac{\pi}{n} r^{2n} e^{-\varphi(0)}.
\end{equation}
\end{lemma}

\begin{proof}
After a unitary change of coordinates, we assume that $\xi=(1,0,\cdots,0)$. We denote by $\hat{g}\in\calO(\DD(0,r))$ the unique minimizer for the infimum \eqref{Eq:Inf1}.

By Guan-Zhou's optimal $L^2$ extension theorem (Corollary \ref{Cor:OptExt3}) and Montel's theorem, such $\hat{g}$ always exists and satisfies the estimate
\begin{equation*}
\int_{\DD(0,r)} |\hat{g}(w)|^2 (r^2-|w|^2)^{n-1} e^{-\varphi(w,0,\cdots,0)} d\lambda \leq \frac{\pi}{n} r^{2n} e^{-\varphi(0)}.
\end{equation*}
By Guan-Zhou's optimal $L^2$ extension theorem again (Corollary \ref{Cor:OptExt2}), there exists a function $G\in\calO(\BB^n(0,r))$ such that $G(w,0,\cdots,0)=\hat{g}(w)$ on $\DD(0,r)$ and
\begin{align}
\label{Eq:pos1} \int_{\BB^n(0,r)} |G|^2e^{-\varphi}d\lambda & \leq \frac{\pi^{n-1}}{(n-1)!} \int_{\DD(0,r)} |\hat{g}(w)|^2 (r^2-|w|^2)^{n-1} e^{-\varphi(w,0,\cdots,0)} d\lambda \\
\label{Eq:pos2} & \leq \frac{\pi^n}{n!} r^{2n} e^{-\varphi(0)}.
\end{align}

As $\BB^n(0,r) = \Omega_{2n\log r}$ and $G(0)=\hat{g}(0)=1$, equality \eqref{Eq:It} yields that
\begin{equation*}
	\int_{\BB^n(0,r)} |G|^2e^{-\varphi}d\lambda \geq I(2n\log r) = \frac{\pi^n}{n!} r^{2n} e^{-\varphi(0)}.
\end{equation*}
Consequently, the inequalities in \eqref{Eq:pos1} and \eqref{Eq:pos2} are indeed equalities. That is
\begin{equation*}
	\int_{\BB^n(0,r)} |G|^2e^{-\varphi}d\lambda = \frac{\pi^n}{n!} r^{2n} e^{-\varphi(0)}
\end{equation*}
and
\begin{equation*} 
	\int_{\DD(0,r)} |\hat{g}(w)|^2 (r^2-|w|^2)^{n-1} e^{-\varphi(w,0,\cdots,0)} d\lambda = \frac{\pi}{n} r^{2n} e^{-\varphi(0)}.
\end{equation*}
In particular, $G\in\calO(\BB^n(0,r))$ is a minimizer for the infimum defining $I(2n\log r)$. By uniqueness of the minimizer and Lemma \ref{Lemma1}, we have
\begin{equation*}
	G \equiv F_0|_{\BB^n(0,r)}.
\end{equation*}
Therefore, $\hat{g}(w) = G(w,0,\cdots,0) = F_0(w,0,\cdots,0)$ for any $w\in\DD(0,r)$, i.e.,
\begin{equation*}
\hat{g}=F_\xi|_{\DD(0,r)}.
\end{equation*}
This completes the proof.
\end{proof}

\begin{lemma} \label{Lemma3}
For any $\xi\in\Sp^{2n-1}$, we have
\begin{equation*}
\int_\DD |F_\xi(w)|^2 e^{-\varphi_\xi(w)} d\lambda = \pi e^{-\varphi(0)}.
\end{equation*}
\end{lemma}

\begin{proof}
For $0\leq s\leq1$, we define
\begin{equation*}
	\alpha(s) := \int_0^{2\pi} |F_\xi(se^{i\theta})|^2 e^{-\varphi_\xi(se^{i\theta})} d\theta.
\end{equation*}
Then equality \eqref{Eq:Fxi} is equivalent to
\begin{equation*}
	\int_0^r \alpha(s) (r^2-s^2)^{n-1}s ds = 2\pi e^{-\varphi(0)} \cdot \frac{r^{2n}}{2n}, \quad 0<r\leq 1.
\end{equation*}
Since
\begin{equation*}
\int_0^r (r^2-s^2)^{n-1}s ds = \frac{r^{2n}}{2n},
\end{equation*} 
the equation can be reformulated as
\begin{equation*}
	\int_0^r \big(\alpha(s) - 2\pi e^{-\varphi(0)}\big) (r^2-s^2)^{n-1} sds = 0, \quad 0<r\leq 1.
\end{equation*}

By the definition of $\alpha(s)$ and equality \eqref{Eq:Fxi}, for any $0<r<1$ we have
\begin{align*}
\int_0^r |\alpha(s)| sds & =
\int_{\DD(0,r)} |F_\xi(w)|^2 e^{-\varphi_\xi(w)} d\lambda \\
& \leq \frac{1}{(1-r^2)^{n-1}} \int_\DD |F_\xi(w)|^2 (1-|w|^2)^{n-1} e^{-\varphi_\xi(w)} d\lambda < +\infty.
\end{align*}
Applying Corollary \ref{Cor:RA} with $\eta(s) = \alpha(s) - 2\pi e^{-\varphi(0)}$, we conclude that
\begin{equation*}
	\alpha(s) = 2\pi e^{-\varphi(0)}
\end{equation*}
for almost every $s\in(0,1)$. Therefore,
\begin{equation*}
	\int_\DD |F_\xi(w)|^2 e^{-\varphi_\xi(w)} d\lambda = \int_0^1 \alpha(s)sds = \pi e^{-\varphi(0)}.
\end{equation*}
This completes the proof.
\end{proof}

\begin{lemma} \label{Lemma4}
For any $\xi\in\Sp^{2n-1}$, the function $F_\xi$ is the unique minimizer for the following infimum:
\begin{equation}
	\inf\left\{ \int_\DD |f(w)|^2 e^{-\varphi_\xi(w)} d\lambda: f\in\calO(\DD), f(0)=1 \right\}.
\end{equation}
\end{lemma}

\begin{proof}
By Lemma \ref{Lemma:Ortho}, it is sufficient to show that
\begin{equation*}
	\int_\DD F_\xi(w) \overline{h(w)} e^{-\varphi_\xi(w)} d\lambda = 0
\end{equation*}
for any $h\in\calO(\DD)$ with $h(0)=0$ and
\begin{equation*}
\int_\DD |h|^2e^{-\varphi_\xi} d\lambda <+\infty.
\end{equation*}

According to Lemma \ref{Lemma2}, the function $F_\xi|_{\DD(0,r)}$ is the unique minimizer for the infimum \eqref{Eq:Inf1}. Since
\begin{equation*}
	\int_{\DD(0,r)} |h(w)|^2 (r^2-|w|^2)^{n-1} e^{-\varphi_\xi(w)} d\lambda \leq \int_\DD |h|^2e^{-\varphi_\xi} d\lambda < +\infty,
\end{equation*}
it follows from the minimal property of $F_\xi|_{\DD(0,r)}$ that
\begin{equation} \label{Eq:Pair}
	\int_{\DD(0,r)} F_\xi(w) \overline{h(w)} (r^2-|w|^2)^{n-1} e^{-\varphi_\xi(w)} d\lambda = 0, \quad 0<r\leq1.
\end{equation}
For $0\leq s\leq1$, we define
\begin{equation*}
	\beta(s) := \int_0^{2\pi} F_\xi(se^{i\theta}) \overline{h(se^{i\theta})} e^{-\varphi_\xi(se^{i\theta})} d\theta.
\end{equation*}
Then equality \eqref{Eq:Pair} is equivalent to
\begin{equation*}
	\int_0^r \beta(s) (r^2-s^2)^{n-1}s ds = 0, \quad 0<r\leq 1.
\end{equation*}

For any $0<r<1$, it is clear that
\begin{align*}
	\int_0^r |\beta(s)| sds \leq &\, \int_0^r \left( \int_0^{2\pi} |F_\xi(se^{i\theta})| |h(se^{i\theta})| e^{-\varphi_\xi(se^{i\theta})} d\theta \right) sds \\
	= &\, \int_{\DD(0,r)} |F_\xi(w)| |h(w)| e^{-\varphi_\xi(w)} d\lambda \\
	\leq &\, \left( \int_{\DD(0,r)} |F_\xi|^2 e^{-\varphi_\xi} d\lambda \right)^{1/2} \left(  \int_{\DD(0,r)} |h|^2 e^{-\varphi_\xi} d\lambda \right)^{1/2} < +\infty.
\end{align*}
Applying Corollary \ref{Cor:RA} to $\beta(s)$, we conclude that $\beta(s)=0$ for almost every $s\in(0,1)$. Consequently,
\begin{equation*}
	\int_\DD F_\xi(w) \overline{h(w)} e^{-\varphi_\xi(w)} d\lambda = \int_0^1 \beta(s) sds = 0.
\end{equation*}
This completes the proof.
\end{proof}

Combining the conclusions of Lemma \ref{Lemma3} and \ref{Lemma4}, we see that
\begin{align*}
\inf\left\{ \int_\DD |f|^2 e^{-\varphi_\xi} d\lambda : f\in\calO(\DD), f(0)=1 \right\} \\
= \int_\DD |F_\xi|^2 e^{-\varphi_\xi} d\lambda = \pi e^{-\varphi(0)}.
\end{align*}
Therefore,
\begin{equation*}
K_\DD(0;e^{-\varphi_\xi}) = \frac{1}{\pi} e^{\varphi_\xi(0)}.
\end{equation*}
According to Theorem \ref{Thm:Harm}, $\varphi_\xi(w) = \varphi(w\xi)$ is a harmonic function on $\DD$.

If $\varphi$ is smooth near $0$, then we can complete the proof by applying Forelli's theorem (Theorem \ref{Thm:Forelli}). However, we need further arguments in the general case.

\begin{lemma}
We have $\varphi \equiv \varphi(0) + 2\log|F_0|$ on $\BB^n$.
\end{lemma}

\begin{proof}
Since $\varphi_\xi$ is harmonic on $\DD$, there exists a holomorphic function $u_\xi\in\calO(\DD)$ such that $\varphi_\xi=2\operatorname{Re}u_\xi$. We consider the Taylor expansion of $F_\xi e^{-u_\xi}\in\calO(\DD)$:
\begin{equation*}
	F_\xi(w)e^{-u_\xi(w)} = \sum_{j=0}^{+\infty} a_j w^j,
\end{equation*}
which is uniformly convergent on any compact subset of $\DD$. Note that
\begin{equation*}
	|a_0|^2 = |F_\xi(0)e^{-u_\xi(0)}|^2 = e^{-\varphi_\xi(0)} = e^{-\varphi(0)}.
\end{equation*}

Since $\{w^j\}_{j=0}^\infty$ is an orthogonal family on $\DD$, it is clear that
\begin{align*}
\int_\DD |F_\xi|^2 e^{-\varphi_\xi} d\lambda & = \int_\DD |F_\xi(w)e^{-u_\xi(w)}|^2 d\lambda \\
& = \int_\DD \left(\sum_{j=0}^{+\infty} a_jw^j\right) \left(\sum_{k=0}^{+\infty} \overline{a_k w^k}\right) d\lambda \\
& = \sum_{j=0}^{+\infty} |a_j|^2 \int_\DD |w|^{2j} d\lambda = \sum_{j=0}^{+\infty} \frac{\pi}{j+1}|a_j|^2.
\end{align*}
On the other hand,
\begin{equation*}
	\int_\DD |F_\xi|^2 e^{-\varphi_\xi} d\lambda = \pi e^{-\varphi(0)} = \pi|a_0|^2.
\end{equation*}
It follows that $a_j=0$ for all $j\geq1$. In other words,
\begin{equation*}
	F_\xi(w)e^{-u_\xi(w)} \equiv a_0 \quad\text{on}\quad \DD.
\end{equation*}

As a consequence,
\begin{equation*}
	|F_0(w\xi)|^2e^{-\varphi(w\xi)} = |F_\xi(w)e^{-u_\xi(w)}|^2 = e^{-\varphi(0)} \quad\text{on}\quad \DD.
\end{equation*}
Since $\xi\in\Sp^{2n-1}$ is arbitrary, we see that
\begin{equation*}
|F_0|^2e^{-\varphi}\equiv e^{-\varphi(0)} \quad\text{on}\quad \BB^n.
\end{equation*}
This completes the proof.
\end{proof}

As $\varphi>-\infty$ everywhere, we know $F_0\neq0$ everywhere, and hence $\varphi$ is pluriharmonic on $\BB^n$. This completes the proof of Theorem \ref{Thm:CharPH}.

\section{Proof of Theorem \ref{Thm:Forelli-type}} \label{Sec4}

We begin with the special case where $\varphi$ is plurisubharmonic on the whole $\BB^n$. In this case, it suffices to assume only that \textit{almost every} slice of $\varphi$ is harmonic.

\begin{theorem} \label{Thm:Forelli-Sim}
	Let $\varphi>-\infty$ be a plurisubharmonic function on $\BB^n$. Assume there exists a set $E\subset\Sp^{2n-1}$ of zero measure such that the slice $w\mapsto\varphi(w\xi)$ is harmonic on $\DD$ for every $\xi\in\Sp^{2n-1}\setminus E$. Then $\varphi$ is pluriharmonic on $\BB^n$.
\end{theorem}

\begin{proof}
Fix a function $F\in A^2(\BB^n;e^{-\varphi})$ with $F(0)=1$, such a function exists by Corollary \ref{Cor:OptExt1}. We choose a vector $\xi\in\Sp^{2n-1}$ such that $w\mapsto\varphi(w\xi)$ is harmonic on $\DD$. Since $\log|F(w\xi)|$ and $-\varphi(w\xi)$ are subharmonic functions, and $\exp$ is a convex increasing function on $\RR$, it follows that
\begin{equation*}
	|F(w\xi)|^2e^{-\varphi(w\xi)} = \exp(2\log|F(w\xi)|-\varphi(w\xi))
\end{equation*}
is also a subharmonic function on $\DD$. Then the mean-value inequality gives
\begin{equation*}
	\frac{1}{2\pi} \int_0^{2\pi} |F(re^{i\theta}\xi)|^2e^{-\varphi(re^{i\theta}\xi)} d\theta \geq e^{-\varphi(0)}, \quad 0<r<1.
\end{equation*}
By assumption, the above inequality holds for every $\xi\in\Sp^{2n-1}\setminus E$.

For any non-negative integrable function $Q$ on $\BB^n$, we have
\begin{align*}
\int_{\BB^n} Q(z) d\lambda_z & = \int_0^1 \left( \int_{\Sp^{2n-1}} Q(r\xi) dS_\xi \right) r^{2n-1} dr \\
& = \int_0^1 \left( \frac{1}{2\pi} \int_0^{2\pi} \left( \int_{\Sp^{2n-1}} Q(re^{i\theta}\xi) dS_\xi \right) d\theta \right) r^{2n-1} dr \\
& = \int_0^1 \left( \int_{\Sp^{2n-1}} \left( \frac{1}{2\pi} \int_0^{2\pi} Q(re^{i\theta}\xi) d\theta \right) dS_\xi \right) r^{2n-1} dr.
\end{align*}
Since $E$ has zero measure, we apply the above formula with $Q = |F|^2e^{-\varphi}$ and obtain
\begin{align*}
\int_{\BB^n} |F|^2e^{-\varphi} d\lambda & = \int_0^1 \left( \int_{\Sp^{2n-1}} \left( \frac{1}{2\pi} \int_0^{2\pi} |F(re^{i\theta}\xi)|^2e^{-\varphi(re^{i\theta}\xi)} d\theta \right) dS_\xi \right) r^{2n-1} dr \\
& \geq \int_0^1 \left( \int_{\Sp^{2n-1}} e^{-\varphi(0)} dS_\xi \right) r^{2n-1} dr = \frac{\pi^n}{n!} e^{-\varphi(0)}.
\end{align*}

As a consequence,
\begin{equation*}
K_{\BB^n}(0;e^{-\varphi}) = \sup_F \frac{|F(0)|^2}{\int_{\BB^n} |F|^2e^{-\varphi} d\lambda} \leq \frac{n!}{\pi^n} e^{\varphi(0)}.
\end{equation*}
Since $\varphi>-\infty$ is plurisubharmonic on $\BB^n$, inequality \eqref{Eq:KerLB} also holds. Therefore, we obtain the equality
\begin{equation*}
K_{\BB^n}(0;e^{-\varphi}) = \frac{n!}{\pi^n} e^{\varphi(0)}.
\end{equation*}
According to Theorem \ref{Thm:CharPH}, $\varphi$ is pluriharmonic on $\BB^n$. This completes the proof.
\end{proof}

Next, we prove the following result which contains Theorem \ref{Thm:Forelli-type} as a corollary.

\begin{theorem} \label{Thm:Forelli-full}
Let $\varphi$ be a real-valued function on $\BB^n$. Assume there exists a set $E\subset\Sp^{2n-1}$ of zero measure such that $w\mapsto\varphi(w\xi)$ is harmonic on $\DD$ for every $\xi\in\Sp^{2n-1}\setminus E$. If $\varphi$ is plurisubharmonic in a neighborhood of $0$, then there exists a pluriharmonic function $\widetilde{\varphi}$ on $\BB^n$ such that $\widetilde{\varphi}=\varphi$ on $\xi\DD$ for every $\xi\in\Sp^{2n-1}\setminus E$.
\end{theorem}

The proof will be a combination of Theorem \ref{Thm:Forelli-Sim} and Forelli's \cite{Forelli} ideas. For this purpose, we need the following variant of Hartogs' lemma (for example, see \cite[Theorem 1.6.13]{Hormander}).

\begin{lemma}\label{Lemma:Hartogs}
Let $v_k$ be a sequence of plurisubharmonic functions on $\Omega\subset\CC^n$ which are uniformly bounded from above on every compact subset of $\Omega$, and assume that 
\begin{equation*}
\limsup_{k \to +\infty} v_k(z) \leq C
\end{equation*}
for \emph{almost every} $z \in \Omega$. For every $\varepsilon > 0$ and every compact set $K \subset \Omega$, one can then find an integer $k_0$ such that
\begin{equation*}
v_k(z) \leq C + \varepsilon, \quad \forall z \in K, \, \forall k > k_0.
\end{equation*}
\end{lemma}

\begin{proof}[Proof of Theorem \ref{Thm:Forelli-full}]
Assume that $\varphi>-\infty$ is plurisubharmonic on $\BB^n(0,\rho)$, where $0<\rho<1$. By a scaling argument, Theorem \ref{Thm:Forelli-Sim} yields that $\varphi$ is pluriharmonic on $\BB^n(0,\rho)$. Then there exists a holomorphic function $f\in\calO(\BB^n(0,\rho))$ such that $\varphi = 2\Re f$.

We consider the Taylor expansion of $f$,
\begin{equation*}
	f(z) = \sum_\alpha c_\alpha z^\alpha,
\end{equation*}
which is compactly convergent on $\BB^n(0,\rho)$. For each integer $k\geq0$, we define
\begin{equation*}
	f_k(z) = \sum_{|\alpha|=k} c_\alpha z^\alpha,
\end{equation*}
which is a homogeneous polynomial of degree $k$.

Fix a vector $\xi\in\Sp^{2n-1}\setminus E$ and set $\varphi_\xi(\tau) = \varphi(\tau\xi)$. Since $\varphi_\xi$ is harmonic on $\DD$, there exists a holomorphic function $g_\xi\in\calO(\DD)$ such that $\varphi_\xi = 2\Re g_\xi$. After adding a purely imaginary constant, we have $g_\xi(\tau) = f(\tau\xi)$ on $\DD(0,\rho)$. Then
\begin{equation*}
	g_\xi(\tau) = \sum_\alpha c_\alpha (\tau\xi)^\alpha = \sum_k f_k(\tau\xi) = \sum_k f_k(\xi)\tau^k
\end{equation*}
with compact convergence on $\DD$. Therefore,
\begin{equation} \label{Eq:SliceEqual}
	\varphi_\xi(\tau) = \sum_{k=0}^{+\infty} f_k(\tau\xi) + \sum_{k=0}^{+\infty} \overline{f_k(\tau\xi)}
\end{equation}
with compact convergence on $\DD$. Following Forelli's \cite{Forelli} ideas, it remains to show that the series $\sum f_k$ is compactly convergent on $\BB^n$.

Since $f$ is holomorphic on $\BB^n(0,\rho)$, we have
\begin{equation*}
	f_k(z) = \frac{1}{2\pi} \int_0^{2\pi} f(e^{i\theta}z) e^{-ik\theta} d\theta, \quad z\in\BB^n(0,\rho).
\end{equation*}
As $f$ is bounded on $\BB^n(0,\rho/2)$, it follows that there is a constant $\beta > 1$ such that
\begin{equation*}
	|f_k(z)| \leq \beta, \quad z\in\BB^n(0,\rho/2).
\end{equation*}
Since $f_k$ is homogeneous, for $k\geq1$ and $z \in \BB^n$, we have
\begin{equation} \label{Eq:Hartogs-Cond1}
	|f_k(z)|^{1/k} = \frac{2}{\rho} \left|f_k\left(\frac{\rho}{2}z\right)\right|^{1/k} \leq \frac{2}{\rho}\beta^{1/k} \leq \frac{2}{\rho}\beta =: \alpha.
\end{equation}

For $\xi \in \Sp^{2n-1}\setminus E$, the power series $\sum f_k(\xi)\tau^k$ converges to $g_\xi$ on $\DD$. Hence,
\begin{equation*}
	\limsup_{k \to +\infty} |f_k(\xi)|^{1/k} \leq 1.
\end{equation*}
By homogeneity, the same inequality also holds for $\tau\xi$, where $\tau\in\DD$. Since $E$ has zero measure, it follows that
\begin{equation} \label{Eq:Hartogs-Cond2}
	\limsup_{k \to +\infty} |f_k(z)|^{1/k} \leq 1, \quad \text{a.e. } z\in\BB^n.
\end{equation}

In the following, we fix constants $0<r_1<r_2<r_3<1$.  Since $f_k$ are holomorphic, $|f_k|^{1/k}$ are plurisubharmonic. By \eqref{Eq:Hartogs-Cond1}, \eqref{Eq:Hartogs-Cond2} and Hartogs' lemma (Lemma \ref{Lemma:Hartogs}), there is an integer $k_0$ such that
\begin{equation*}
	|f_k(z)|^{1/k} \leq \frac{r_3}{r_2}, \quad \forall z\in\BB^n(0,r_3), \, \forall k > k_0.
\end{equation*}
Since $f_k$ is homogeneous,
\begin{equation*}
	|f_k(z)| \leq 1, \quad \forall z\in\BB^n(0,r_2), \, \forall k > k_0.
\end{equation*}
As a consequence, there is a constant $C>0$ such that
\begin{equation*}
	|f_k(z)| \leq C, \quad \forall z\in\BB^n(0,r_2), \, \forall k \geq 0.
\end{equation*}
Using homogeneity again, we get
\begin{equation*}
	|f_k(z)|\leq C\left(\frac{r_1}{r_2}\right)^k, \quad \forall z\in\BB^n(0,r_1), \, \forall k \geq 0.
\end{equation*}
Therefore, $\sum f_k$ converges compactly on $\BB^n(0,r_1)$.

Since $0<r_1<1$ is arbitrary, we conclude that $\sum f_k$ converges compactly to a holomorphic function $\widetilde{f}$ on $\BB^n$. Then $\widetilde{\varphi} := 2\Re \widetilde{f}$ is a pluriharmonic function on $\BB^n$. According to \eqref{Eq:SliceEqual}, we have $\widetilde{\varphi}=\varphi$ on $\xi\DD$ for every $\xi\in\Sp^{2n-1}\setminus E$. 
\end{proof}

\section{Proof of Theorem \ref{Thm:CharPH-pKer}} \label{Sec5}

To prove Theorem \ref{Thm:CharPH-pKer}, we need a variant of Theorem \ref{Thm:CharPH}.

\begin{lemma} \label{Lemma6}
Let $\varphi>-\infty$ be a subharmonic function on $\DD$. Let $c>0$ be a constant and $h$ a holomorphic function on $\DD$ with $h(0)=1$. If
\begin{equation*}
K_{\DD}(0;e^{-\varphi-c\log|h|}) = \frac{1}{\pi} e^{\varphi(0)},
\end{equation*}
then $\varphi$ is harmonic on $\DD$.
\end{lemma}

\begin{proof}
This lemma is a corollary to Theorem 1.11 of Guan-Mi \cite{GuanMi22}. Indeed, their result is much stronger, but its proof is rather lengthy. For convenience, we give a short proof of what we claimed.

Actually, we only need a slight modification of the proof of Theorem \ref{Thm:Harm} (see pages 20-22 of \cite{XuZhou24} for its proof). Suppose, to the contrary, that $\varphi>-\infty$ is \textit{not} harmonic. Then there exists a subharmonic function $\widetilde{\varphi} \geq \varphi$ on $\DD$ such that $\widetilde{\varphi}\equiv\varphi$ on $\DD\setminus\DD(x;r)$ and $\widetilde{\varphi}\not\equiv\varphi$ on $\DD(x;r)$, where $\overline{\DD(x,r)} \subset \DD$. Applying exactly the same arguments to $\varphi+c\log|h|$ and $\widetilde{\varphi}+c\log|h|$ then yields a contradiction. Therefore, $\varphi$ must be harmonic.
\end{proof}

\begin{lemma} \label{Lemma7}
Let $\varphi>-\infty$ be a plurisubharmonic function on $\BB^n$, where $n\geq2$. Let $c>0$ be a constant and $h$ a holomorphic function on $\BB^n$ with $h(0)=1$. If
\begin{equation*}
K_{\BB^n}(0;e^{-\varphi-c\log|h|}) = \frac{n!}{\pi^n} e^{\varphi(0)},
\end{equation*}
then $\varphi$ is pluriharmonic on $\BB^n$.
\end{lemma}

\begin{proof}
The proof is a slight modification of the proof of Theorem \ref{Thm:CharPH}. As before, let $\Omega_t = \{\psi<t\}$ be the sublevel set of $\psi$ and let
\begin{equation*}
	I(t) = \inf\left\{ \int_{\Omega_t} |F|^2e^{-\varphi-c\log|h|}d\lambda: F\in\calO(\Omega_t),F(0)=1 \right\}
\end{equation*}
be the minimal $L^2$ integral. We denote by $F_t\in\calO(\Omega_t)$ the unique minimizer of the above infimum. For any $\xi\in\Sp^{2n-1}$, we define
\begin{equation*}
F_\xi(w)=F_0(w\xi), \quad \varphi_\xi(w)=\varphi(w\xi), \quad h_\xi(w)=h(w\xi).
\end{equation*}
It is clear that, with $\varphi$ replaced by $\varphi+c\log|h|$ and $\varphi_\xi$ replaced by $\varphi_\xi+c\log|h_\xi|$, Lemma \ref{Lemma1}, \ref{Lemma2}, \ref{Lemma3}, \ref{Lemma4} remain true. Combining the conclusions of Lemma \ref{Lemma3} and \ref{Lemma4}, we see that
\begin{align*}
\inf\left\{ \int_\DD |f|^2 e^{-\varphi_\xi-c\log|h_\xi|} d\lambda : f\in\calO(\DD), f(0)=1 \right\} \\ = \int_\DD |F_\xi|^2 e^{-\varphi_\xi-c\log|h_\xi|} d\lambda = \pi e^{-\varphi(0)}.
\end{align*}
Therefore,
\begin{equation*}
	K_\DD(0;e^{-\varphi_\xi-c\log|h_\xi|}) = \frac{1}{\pi} e^{\varphi_\xi(0)}.
\end{equation*}
According to Lemma \ref{Lemma6}, $\varphi_\xi(w) = \varphi(w\xi)$ is a harmonic function on $\DD$. Now, it follows from Theorem \ref{Thm:Forelli-Sim} that $\varphi$ is pluriharmonic on $\BB^n$.
\end{proof}

Now, we are ready to prove Theorem \ref{Thm:CharPH-pKer} in the case of $z=0$.

\begin{proof}[Proof of Theorem \ref{Thm:CharPH-pKer} (in the case of $z=0$)]
Firstly, we prove the sufficiency part.

Since $\varphi$ is pluriharmonic, there exists $u\in\calO(\BB^n)$ such that $\varphi=p\Re u$, and hence $|e^u|^p=e^{\varphi}$. Clearly, $f:=e^{u-u(0)}$ is a holomorphic function on $\BB^n$ such that $f(0)=1$ and
\begin{equation*}
	\int_{\BB^n}|f|^pe^{-\varphi}d\lambda = \int_{\BB^n}|e^{-u(0)}|^pd\lambda = \frac{\pi^n}{n!} e^{-\varphi(0)}.
\end{equation*}
On the other hand, since $-\varphi$ is plurisubharmonic, for any $g\in\calO(\BB^n)$, we know
\begin{equation*}
	|g|^pe^{-\varphi} = \exp(-\varphi+p\log|g|)
\end{equation*}
is also a plurisubharmonic function. Then the mean value inequality says that
\begin{equation*}
	\int_{\BB^n}|g|^pe^{-\varphi}d\lambda \geq \frac{\pi^n}{n!} |g(0)|^pe^{-\varphi(0)}.
\end{equation*}
As a consequence,
\begin{equation*}
	K_{\BB^n,p}(0;e^{-\varphi}) = \sup_g \frac{|g(0)|^p}{\int_{\BB^n}|g|^pe^{-\varphi}d\lambda} = \frac{n!}{\pi^n} e^{\varphi(0)}.
\end{equation*}	

Next, we prove the necessity part. To that end, we assume that
\begin{equation*}
K_{\BB^n,p}(0;e^{-\varphi}) = \frac{n!}{\pi^n} e^{\varphi(0)}.
\end{equation*}

If there exists a positive integer $m$ such that $p=2/m$, then
\begin{align*}
K_{\BB^n}(0;e^{-\varphi}) & = \sup_f \frac{|f(0)|^2}{\int_{\BB^n}|f|^2e^{-\varphi}d\lambda} \\
& = \sup_f \frac{|f^m(0)|^p}{\int_{\BB^n}|f^m|^pe^{-\varphi}d\lambda} \leq K_{\BB^n,p}(0;e^{-\varphi}).
\end{align*}
Then it follows from inequality \eqref{Eq:KerLB} and the assumption that
\begin{equation*}
\frac{n!}{\pi^n} e^{\varphi(0)} \leq K_{\BB^n}(0;e^{-\varphi}) \leq K_{\BB^n,p}(0;e^{-\varphi}) = \frac{n!}{\pi^n} e^{\varphi(0)}.
\end{equation*}
Consequently,
\begin{equation*}
K_{\BB^n}(0;e^{-\varphi}) = \frac{n!}{\pi^n} e^{\varphi(0)}.
\end{equation*}
According to Theorem \ref{Thm:CharPH}, $\varphi$ is a pluriharmonic function on $\BB^n$. This proves the theorem in the case that $p=2/m$.

For general $p\in(0,2)$, we need a different argument. We assume the contrary, i.e., $\varphi$ is plurisubharmonic but \textit{not} pluriharmonic.

By Guan-Zhou's optimal $L^p$ extension theorem (Theorem \ref{Thm:OptLpExt}), there exists a holomorphic function $h\in\calO(\BB^n)$ such that $h(0)=1$ and
\begin{equation*}
	\int_{\BB^n}|h|^pe^{-\varphi}d\lambda \leq \frac{\pi^n}{n!} e^{-\varphi(0)}.
\end{equation*}
Since $\varphi$ is not pluriharmonic, Lemma \ref{Lemma7} yields
\begin{equation*}
K_\BB^n(0;e^{-\varphi-(2-p)\log|h|}) > \frac{n!}{\pi^n} e^{\varphi(0)}.
\end{equation*}
Equivalently, there exists $f\in\calO(\BB^n)$ such that $f(0)=1$ and
\begin{equation*}
	\int_{\BB^n} |f|^2e^{-\varphi-(2-p)\log|h|} d\lambda < \frac{\pi^n}{n!} e^{-\varphi(0)}.
\end{equation*}
By H\"older's inequality,
\begin{align*}
	\int_{\BB^n} |f|^pe^{-\varphi} d\lambda & = \int_{\BB^n} \big(|h|^pe^{-\varphi}\big)^{\frac{2-p}{2}} \big(|f|^2e^{-\varphi-(2-p)\log|h|}\big)^{\frac{p}{2}} d\lambda \\
	& \leq \left(\int_{\BB^n}|h|^pe^{-\varphi}d\lambda\right)^{\frac{2-p}{2}} \left(\int_{\BB^n}|f|^2e^{-\varphi-(2-p)\log|h|}d\lambda\right)^{\frac{p}{2}} \\
	& < \left(\frac{\pi^n}{n!} e^{-\varphi(0)}\right)^{\frac{2-p}{2}} \left(\frac{\pi^n}{n!} e^{-\varphi(0)}\right)^{\frac{p}{2}} = \frac{\pi^n}{n!} e^{-\varphi(0)}.
\end{align*}
It follows that
\begin{equation*}
K_{\BB^n,p}(0;e^{-\varphi}) \geq \frac{|f(0)|^p}{\int_{\BB^n} |f|^pe^{-\varphi} d\lambda} > \frac{n!}{\pi^n} e^{\varphi(0)}.
\end{equation*}
This contradicts the assumption. Hence, $\varphi$ must be pluriharmonic.
\end{proof}

To prove the general case of Theorem \ref{Thm:CharPH-pKer}, we need a transformation rule for weighted $p$-Bergman kernels. By an argument similar to that of Ning-Zhang-Zhou \cite[Proposition 2.1]{NZZ16}, we deduce the following two formulae.

\begin{lemma}
Let $\Phi:\Omega_1\to\Omega_2$ be a biholomorphic map between two domains in $\CC^n$. Let $\varphi$ be an upper semi-continuous function on $\Omega_2$ and $p>0$ a constant. Then
\begin{equation} \label{Eq:Rule1}
K_{\Omega_2,p}(\Phi(z);e^{-\varphi}) = K_{\Omega_1,p}(z;e^{-\varphi\circ\Phi+2\log|J_\Phi|}), \quad z\in\Omega_1.
\end{equation}
Moreover, if $\Omega_1$ is simply connected or $p=2/m$ for some integer $m$, then
\begin{equation} \label{Eq:Rule2}
K_{\Omega_2,p}(\Phi(z);e^{-\varphi}) = K_{\Omega_1,p}(z;e^{-\varphi\circ\Phi}) / |J_\Phi(z)|^2, \quad z\in\Omega_1.
\end{equation}
Here, $J_\Phi$ is the determinant of the holomorphic Jacobian matrix of $\Phi$.
\end{lemma}

\begin{proof}
Since $\Phi$ is biholomorphic, $J_\Phi$ is a non-vanishing holomorphic function on $\Omega_1$. Given $f\in A^p(\Omega_2,e^{-\varphi})$, it is clear that
\begin{align*}
\int_{\Omega_2} |f(w)|^pe^{-\varphi(w)} d\lambda & = \int_{\Omega_1} |f(\Phi(z))|^pe^{-\varphi(\Phi(z))} |J_\Phi(z)|^2 d\lambda \\
& = \int_{\Omega_1} |f(\Phi(z))|^p e^{-\varphi(\Phi(z))+2\log|J_\Phi(z)|} d\lambda.
\end{align*}
Therefore,
\begin{equation*}
\Phi^*f = f\circ\Phi \in A^p(\Omega_1, e^{-\widetilde{\varphi}}),
\end{equation*}
where
\begin{equation*}
\widetilde{\varphi} := \varphi\circ\Phi-2\log|J_\Phi|.
\end{equation*}
Conversely, if $\Phi^*f \in A^p(\Omega_1, e^{-\widetilde{\varphi}})$, then $f\in A^p(\Omega_2,e^{-\varphi})$. With $w=\Phi(z)$, we have
\begin{align*}
K_{\Omega_2,p}(w;e^{-\varphi})
& = \sup_f \frac{|f(w)|^p}{\int_{\Omega_2} |f|^pe^{-\varphi} d\lambda} \\
& = \sup_{\Phi^*f} \frac{|f(\Phi(z))|^p}{\int_{\Omega_1} |\Phi^*f|^pe^{-\widetilde{\varphi}} d\lambda} = K_{\Omega_1,p}(z;e^{-\widetilde{\varphi}}).
\end{align*}

If $\Omega_1$ is simply connected, then there exists a holomorphic function $g$ such that $e^g = J_\Phi$. In other words, $g$ is a single-valued branch of $\log J_\Phi$. Then $h:=e^{2g/p}$ is a holomorphic function on $\Omega_1$ such that $|h|^p=|J_\Phi|^2$. If $p=2/m$ for some integer $m$, we simply take $h:=J_\Phi^m$.
In either case, for any $f\in A^p(\Omega_2,e^{-\varphi})$, we have
\begin{align*}
\int_{\Omega_2} |f(w)|^pe^{-\varphi(w)} d\lambda = \int_{\Omega_1} |f(\Phi(z))h(z)|^p e^{-\varphi(\Phi(z))} d\lambda.
\end{align*}
Therefore, $h\cdot\Phi^*f \in A^p(\Omega_1,e^{-\varphi\circ\Phi})$. It is clear that the correspondence
\begin{equation*}
f\in A^p(\Omega_2,e^{-\varphi}) \quad\mapsto\quad \widehat{f} := h\cdot\Phi^*f \in A^p(\Omega_1,e^{-\varphi\circ\Phi})
\end{equation*}
is bijective. With $w=\Phi(z)$, we have
\begin{align*}
K_{\Omega_2,p}(w;e^{-\varphi})
& = \sup_f \frac{|f(w)|^p}{\int_{\Omega_2} |f|^pe^{-\varphi} d\lambda} \\
& = \sup_{\widehat{f}} \frac{|\widehat{f}(z)/h(z)|^p}{\int_{\Omega_1} |\widehat{f}|^pe^{-\varphi\circ\Phi} d\lambda} = \frac{K_{\Omega_1,p}(z;e^{-\varphi\circ\Phi})}{|J_\Phi(z)|^2}.
\end{align*}
This completes the proof.
\end{proof}

Next, we recall some facts concerning the holomorphic automorphisms of $\BB^n$. Given a point $a\in\BB^n\setminus\{0\}$, let
\begin{equation*}
p_a(z) := \frac{\inner{z,a}}{|a|^2} a \quad\text{and}\quad q_a(z) := z - p_a(z).
\end{equation*}
Here, $\inner{z,a} := z_1\overline{a_1} + \cdots + z_n\overline{a_n}$ denotes the standard inner product of $\CC^n$.
It is well-known that
\begin{equation*}
\Phi(z) = \frac{a - p_a(z) - \sqrt{1-|a|^2}q_a(z)}{1-\inner{z,a}}
\end{equation*}
is a holomorphic automorphism of $\BB^n$. Moreover, by direct computations,
\begin{gather*}
\Phi(0)=a, \quad \Phi(a)=0, \quad
J_\Phi(z) = (-1)^n \left(\frac{\sqrt{1-|a|^2}}{1-\inner{z,a}}\right)^{n+1}.
\end{gather*}

Now, we can prove the general case of Theorem \ref{Thm:CharPH-pKer}.

\begin{proof}[Proof of Theorem \ref{Thm:CharPH-pKer} (the general case)]

Let $\varphi>-\infty$ be a plurisubharmonic function on $\BB^n$ and $0<p\leq2$ a constant. Let $a\in\BB^n\setminus\{0\}$ and $\Phi\in\operatorname{Aut}(\BB^n)$ be the same as above. Then inequalities \eqref{Eq:pKerLB} and \eqref{Eq:Rule2} together imply
\begin{align*}
K_{\BB^n,p}(a;e^{-\varphi}) & = K_{\BB^n,p}(0;e^{-\varphi\circ\Phi}) / |J_\Phi(0)|^2 \\
& \geq \frac{n!}{\pi^n} \frac{e^{\varphi(\Phi(0))}}{|J_\Phi(0)|^2} = \frac{n!}{\pi^n} \frac{e^{\varphi(a)}}{(1-|a|^2)^{n+1}}.
\end{align*}
Note that, using \eqref{Eq:Rule1} instead of \eqref{Eq:Rule2} would yield the same inequality. This proves the inequality \eqref{Eq:pKerLB}.

In the following, we suppose that
\begin{equation*}
K_{\BB^n,p}(a;e^{-\varphi}) = \frac{n!}{\pi^n} \frac{e^{\varphi(a)}}{(1-|a|^2)^{n+1}}.
\end{equation*}
Then the above proof shows that
\begin{equation*}
K_{\BB^n,p}(0;e^{-\varphi\circ\Phi}) = \frac{n!}{\pi^n} e^{\varphi(\Phi(0))}.
\end{equation*}
Since Theorem \ref{Thm:CharPH-pKer} holds for the case of $z=0$, we conclude that $\varphi\circ\Phi$ is pluriharmonic. Since $\Phi\in\operatorname{Aut}(\BB^n)$, we know $\varphi$ itself is pluriharmonic. This completes the proof for the general case.
\end{proof}

For the convenience of readers, we record the following optimal $L^p$ extension theorem as an equivalent statement of inequality \eqref{Eq:pKerLB}.

\begin{theorem} \label{Thm:OptLpExt-2}
Let $\phi>-\infty$ be a plurisubharmonic function on $\BB^n$. Let $0<p\leq2$ be a constant. For any $z\in\BB^n$, there exists a holomorphic function $F$ on $\BB^n$ such that $F(z)=1$ and
\begin{equation*}
	\int_{\BB^n} |F|^pe^{-\phi} d\lambda \leq \frac{\pi^n}{n!}(1-|z|^2)^{n+1} e^{-\phi(z)}.
\end{equation*}
\end{theorem}

\end{document}